\documentclass[10pt]{article}
\usepackage[T1]{fontenc}
\usepackage{geometry}
\usepackage{amssymb}
\usepackage{indentfirst}
\usepackage{amsmath}
\usepackage{amsthm}
\usepackage{url}
\newtheorem{theorem}{Theorem}[section]
\newtheorem{proposition}[theorem]{Proposition}
\newtheorem{lemma}[theorem]{Lemma}
\newtheorem*{corollary}{Corollary}
\numberwithin{equation}{section}

\usepackage{bbm}

\newcommand{\bbf}[1]{\mathbf{#1}}
\title{Normal and Lognormal Asymptotics for Lattice-Point Counts in Random Translates of High-Dimensional Balls}
\author{Xin Hang Ji\thanks{Corresponding author. Email:
		\texttt{1000535190@smail.shnu.edu.cn}}\\
	Department of Mathematics, Shanghai Normal University}
\date{}
\begin{document}
	\maketitle
	
	\begin{abstract}
		Let the center of a $d$-dimensional Euclidean ball be uniformly distributed
		modulo $\mathbb Z^d$. We study the resulting number of integer lattice points
		as the dimension and radius tend to infinity. The critical scale is
		$4\pi^2R_d^2/(d+2)=\log d$. In a logarithmic window around this scale, the
		log-ratio of the lattice-point count to the volume satisfies a central limit
		theorem. At a fixed offset from the critical scale, this yields a genuine
		lognormal limit. When
		$4\pi^2R_d^2/(d+2)-\log d\to+\infty$ while
		$4\pi^2R_d^2/(d+2)=o(\sqrt d)$, the count-to-volume ratio converges to $1$
		in measure, its normalized error has a standard normal limit, and its
		variance satisfies an asymptotic formula.
	\end{abstract}
	
	
	
	\section{Introduction}
	
	Lattice-point counting usually fixes the dimension and dilates the region.
	Here the dimension and the radius of the ball vary together, while the center
	ranges over a fundamental domain. We determine the value distribution of the
	resulting lattice-point count as $d\to\infty$.
	
	Throughout, let $d\in\mathbb N$, let $R=R_d>0$, and let $|\cdot|$ denote the
	Euclidean norm. For $\bbf x\in\mathbb R^d$, define
	\[
	f_d(\bbf x)=\#\{\bbf n\in\mathbb Z^d:|\bbf n-\bbf x|\leq R_d\},
	\qquad
	\mu_d=\frac{\pi^{d/2}R_d^d}{\Gamma(1+\frac d2)}.
	\]
	Thus $\mu_d$ is the volume of the ball. Since
	$f_d(\bbf x+\bbf n)=f_d(\bbf x)$ for every $\bbf n\in\mathbb Z^d$, we regard
	$f_d$ as a function on the torus $[0,1)^d$. Unless otherwise specified,
	unqualified measures and $L^p$-norms of functions of the center refer to
	Lebesgue measure on this torus; in particular,
	\[
	\|h\|_2^2=\int_{[0,1)^d}|h(\bbf x)|^2\,\mathrm{d}\bbf x.
	\]
	Decomposing $\mathbb R^d$ into translates of a fundamental domain gives
	\[
	\int_{[0,1)^d}f_d(\bbf x)\,\mathrm{d}\bbf x=\mu_d.
	\]
	
	Variable-center lattice-point problems go back at least to Kendall
	\cite{Kendall1948} and Kendall--Rankin \cite{KendallRankin1953}.
	Vinogradov--Skriganov obtained a lower bound for the oscillation of the count
	as the center of a large sphere varies over a fundamental domain
	\cite{VinogradovSkriganov1981}. Skriganov--Sobolev subsequently obtained
	lower bounds for this variation in large balls, including sharp results for
	rational lattices \cite{SkriganovSobolev2005}. Mazo--Odlyzko
	\cite{MazoOdlyzko1990} fixed
	$\mathbb Z^d$, varied the center, and studied the mean and extrema when
	$R_d^2=\alpha d$ with fixed $\alpha>0$. At that scale
	$4\pi^2R_d^2/(d+2)$ remains bounded. In contrast, the scale considered here
	grows like $\log d$, and the results describe the distribution over the center,
	not only its extreme values.
	
	Bleher--Bourgain fixed the dimension and the center and studied the value
	distribution of the error term as the radius varies
	\cite{BleherBourgain1996}. Kang--Sobolev obtained variance asymptotics for
	the fluctuations in large balls and thin spherical shells centered at a
	Diophantine point \cite{KangSobolev2010}. Str\"ombergsson--S\"odergren
	instead randomized the
	lattice and proved a functional central limit theorem in large dimension
	\cite{StrombergssonSodergren2019}. The results below keep the lattice
	$\mathbb Z^d$ fixed, average only over $\bbf x\in[0,1)^d$, and let the radius
	and dimension approach a scale not covered by those limiting procedures. They
	give the value distribution in this setting and identify a transition from
	lognormal behavior to concentration with asymptotically normal fluctuations.
	
	The shortest nonzero Fourier frequencies determine the relevant radius scale.
	Write
	\[
	\ell_d=\frac{4\pi^2R_d^2}{d+2},
	\qquad
	q_d=e^{-\ell_d/2},
	\qquad
	A_d=\ell_d-\log d.
	\]
	In both regimes treated below, $\ell_d=o(\sqrt d)$. The Fourier coefficient of
	$f_d/\mu_d$ at each frequency $|\bbf m|=1$ is then
	$q_d(1+o(1))$. Since there are $2d$ such frequencies, their squared
	$L^2$ mass is asymptotic to
	\[
	2dq_d^2=2e^{-A_d}.
	\]
	Thus, within the range considered here, the balance between a nonvanishing
	first Fourier layer and a vanishing one occurs when $\ell_d-\log d$ is bounded.
	The logarithm in the
	critical scale comes from the multiplicity $2d$ of the shortest frequencies:
	a single coefficient has squared size $e^{-\ell_d}$, while the whole layer has
	squared size of order
	\[
	dq_d^2=e^{-A_d}.
	\]
	When $A_d$ remains of order one, this layer produces fluctuations of constant
	order, whereas
	$A_d\to+\infty$ forces its $L^2$ mass to zero. The former range is therefore
	governed by $\log(f_d/\mu_d)$, while in the latter one can linearize and study
	$f_d/\mu_d-1$. We refer to these ranges as the
	\emph{logarithmic fluctuation regime} and the \emph{volume-concentration
		regime}, respectively. In the latter regime, we allow
	$\ell_d-\log d\to\infty$ as long as $\ell_d=o(\sqrt d)$.
	
	\subsection{Main results}
	
	Throughout the rest of the paper, write $e(t)=e^{2\pi i t}$.

	\begin{theorem}[Normal limit for the logarithm]\label{thm:logarithmic}
		Suppose that
		\[
		\limsup_{d\to\infty}A_d<\infty,
		\qquad
		e^{-A_d}=o(\log d).
		\]
		Then $\ell_d\to\infty$, and for all sufficiently large $d$ one has
		$f_d(\bbf x)>0$ for every $\bbf x\in[0,1)^d$. For any real numbers $a<b$,
		\begin{equation}\label{eq:logarithmic-clt}
			\lim_{d\to\infty}\operatorname{meas}\left\{\bbf x\in[0,1)^d:
			a<\frac{\log(f_d(\bbf x)/\mu_d)+e^{-A_d}}
			{\sqrt{2e^{-A_d}}}\leq b\right\}
			=\frac{1}{\sqrt{2\pi}}\int_a^b e^{-u^2/2}\,\mathrm{d}u.
		\end{equation}
	\end{theorem}
	
	\begin{corollary}[Lognormal limit]
		If, in addition to the assumptions of Theorem~\ref{thm:logarithmic},
		$A_d\to\tau\in\mathbb R$, then
		\[
		\frac{f_d}{\mu_d}\ \mathrel{\Rightarrow}\
		\operatorname{LN}(-e^{-\tau},2e^{-\tau}).
		\]
		Equivalently, for any $0<\alpha<\beta$,
		\begin{equation}\label{eq:lognormal-corollary}
			\lim_{d\to\infty}\operatorname{meas}\left\{\bbf x\in[0,1)^d:
			\alpha<\frac{f_d(\bbf x)}{\mu_d}\leq\beta\right\}
			=\int_\alpha^\beta
			\frac{1}{u\sqrt{4\pi e^{-\tau}}}
			\exp\left(-\frac{(\log u+e^{-\tau})^2}{4e^{-\tau}}\right)\,\mathrm{d}u.
		\end{equation}
	\end{corollary}
	
	\begin{theorem}[Normal limit in the volume-concentration regime]\label{thm:concentration}
		Suppose that
		\begin{equation}\label{eq:concentration-assumptions}
			A_d\longrightarrow+\infty,
			\qquad
			\log d+A_d=o(\sqrt d).
		\end{equation}
		Then, for any real numbers $a<b$,
		\begin{equation}\label{eq:concentration-clt}
			\lim_{d\to\infty}\operatorname{meas}\left\{\bbf x\in[0,1)^d:
			a<\frac{f_d(\bbf x)/\mu_d-1}{\sqrt{2e^{-A_d}}}\leq b\right\}
			=\frac{1}{\sqrt{2\pi}}\int_a^b e^{-u^2/2}\,\mathrm{d}u.
		\end{equation}
		In particular, $f_d/\mu_d\to1$ in measure.
		Moreover,
		\begin{equation}\label{eq:concentration-var}
			\int_{[0,1)^d}\left|\frac{f_d(\bbf x)}{\mu_d}-1\right|^2
			\,\mathrm{d}\bbf x\sim2e^{-A_d}.
		\end{equation}
	\end{theorem}
	
	Without convergence of $A_d$, Theorem~\ref{thm:logarithmic} is a Gaussian
	approximation for the standardized logarithm; a fixed lognormal law occurs
	when $A_d\to\tau$. The limiting lognormal law above has mean
	$\exp(-e^{-\tau}+\tfrac12(2e^{-\tau}))=1$, in agreement with
	$\int_{[0,1)^d}(f_d/\mu_d)\,\mathrm{d}\bbf x=1$.
	
	The proof compares the Bessel and Gaussian Fourier coefficients in $\ell^2$
	and factors the Gaussian series into one-dimensional $\theta$ functions. The
	$\theta$-series limit requires $dq_d^2=o(\sqrt d)$, while
	$dq_d^2=o(\log d)$ is used only to retain the exponential weight needed before
	taking logarithms.
	
	\section{Proofs}
	
	\subsection{Fourier expansion and Bessel estimates}
	
	The function $f_d$ can be written as
	\[
	f_d(\bbf x)=\sum_{\bbf n\in\mathbb Z^d}
	\mathbbm 1_{\{|\bbf n-\bbf x|\leq R\}}.
	\]
	We use the Fourier transform convention
	\[
	\widehat g(\bbf t)=\int_{\mathbb R^d}
	g(\bbf y)e(-\bbf y\cdot\bbf t)\,\mathrm{d}\bbf y.
	\]
	Here $J_\nu$ denotes the Bessel function of the first kind. A direct
	calculation gives, for $\bbf t\neq\bbf 0$,
	\[
	\int_{\{|\bbf y|\leq R\}}e(-\bbf y\cdot\bbf t)\,\mathrm{d}\bbf y=
	\left(\frac{R}{|\bbf t|}\right)^{d/2}
	J_{d/2}(2\pi R|\bbf t|),
	\]
	whereas the integral at $\bbf t=\bbf 0$ equals $\mu_d$.
	
	The function $f_d$ is bounded on $[0,1)^d$ and therefore belongs to
	$L^2([0,1)^d)$.
	Its Fourier coefficient at $\bbf m\in\mathbb Z^d$, with respect to the basis
	$e(-\bbf m\cdot\bbf x)$, is
	\begin{align*}
		\int_{[0,1)^d}f_d(\bbf x)e(\bbf m\cdot\bbf x)\,\mathrm{d}\bbf x
		&=\sum_{\bbf n\in\mathbb Z^d}
		\int_{[0,1)^d}\mathbbm 1_{\{|\bbf n-\bbf x|\leq R\}}
		e(\bbf m\cdot\bbf x)\,\mathrm{d}\bbf x\\
		&=\int_{\mathbb R^d}\mathbbm 1_{\{|\bbf y|\leq R\}}
		e(-\bbf m\cdot\bbf y)\,\mathrm{d}\bbf y.
	\end{align*}
	Only finitely many $\bbf n$ can contribute for $\bbf x\in[0,1)^d$, so the
	interchange above is immediate. Consequently,
	\begin{equation}\label{eq:unnormalized-fourier-expansion}
		f_d(\bbf x)=\mu_d+
		\sum_{\bbf m\in\mathbb Z^d\setminus\{\bbf 0\}}
		e(-\bbf x\cdot\bbf m)
		\left(\frac{R}{|\bbf m|}\right)^{d/2}
		J_{d/2}(2\pi R|\bbf m|)
	\end{equation}
	in $L^2([0,1)^d)$. All Fourier-series identities below are understood in this
	$L^2$ sense, which is also the sense in which Parseval's identity is used.
	
	Define $\mathcal J_d:\mathbb R\to\mathbb R$ by
	\begin{equation}\label{eq:normalized-bessel}
		\mathcal J_d(x)=
		\begin{cases}
			1,&x=0,\\[2mm]
			\displaystyle
			\left(\frac{2}{|x|}\right)^{d/2}
			\Gamma\left(1+\frac d2\right)J_{d/2}(|x|),&x\neq0.
		\end{cases}
	\end{equation}
	The Fourier expansion above and \eqref{eq:normalized-bessel} give
	\begin{equation}\label{eq:fourier-expansion}
		\frac{f_d(\bbf x)}{\mu_d}
		=\sum_{\bbf m\in\mathbb Z^d}
		e(-\bbf x\cdot\bbf m)\mathcal J_d(2\pi R|\bbf m|).
	\end{equation}
	
	We first collect the Bessel-function estimates needed in the proof. The
	relevant infinite-product representation and bounds for the Bessel zeros are
	recorded in \cite[\S10.21]{DLMF}; related real-axis formulas may be found in
	\cite{Watson1944}. For completeness, the two Rayleigh identities used below
	are derived directly in the proof.
	
	\begin{lemma}\label{lem:bessel}
		There exists an absolute constant $C>0$ such that the following statements
		hold. If $0\leq x\leq d/2$, then
		\begin{equation}\label{eq:bessel-gaussian-bound}
			|\mathcal J_d(x)|\leq
			\exp\left(-\frac{x^2}{2d+4}\right).
		\end{equation}
		If $0\leq x\leq d/4$, then
		\begin{equation}\label{eq:bessel-local-expansion}
			\mathcal J_d(x)=
			\exp\left(-\frac{x^2}{2d+4}
			+O\left(\frac{x^4}{d^3}\right)\right).
		\end{equation}
		The implied constant in this estimate is absolute and uniform in $d$ and $x$.
		For $x>0$,
		\[
		|\mathcal J_d(x)|
		\leq \Gamma\left(1+\frac d2\right)
		\left(\frac2x\right)^{d/2}.
		\]
		Finally, for $x\geq d$,
		\[
		|\mathcal J_d(x)|
		\leq C\Gamma\left(1+\frac d2\right)
		\left(\frac2x\right)^{d/2}x^{-1/2}.
		\]
	\end{lemma}
	
	\begin{proof}
		In this proof, write $\nu=d/2$ and let $j_{\nu,r}$ denote the $r$th
		positive zero of $J_\nu$. The Weierstrass product for the normalized Bessel
		function is
		\[
		\mathcal J_d(x)
		=\prod_{r=1}^{\infty}
		\left(1-\frac{x^2}{j_{\nu,r}^2}\right).
		\]
		The power series for $J_\nu$ gives
		\[
		\Gamma(\nu+1)\left(\frac2x\right)^\nu J_\nu(x)
		=1-\frac{x^2}{4(\nu+1)}
		+\frac{x^4}{32(\nu+1)(\nu+2)}+O(x^6).
		\]
		Expanding the product through order $x^4$ and comparing coefficients yields
		the two Rayleigh identities
		\[
		\sum_{r=1}^{\infty}\frac1{j_{\nu,r}^2}
		=\frac1{4(\nu+1)},
		\qquad
		\sum_{r=1}^{\infty}\frac1{j_{\nu,r}^4}
		=\frac1{16(\nu+1)^2(\nu+2)}.
		\]
		Sturm comparison for the Bessel equation gives $j_{\nu,1}>\nu$.
		Consequently, when $0\leq x\leq\nu$, every factor in the product above is
		positive. The inequality $1-u\leq e^{-u}$ then immediately yields
		\eqref{eq:bessel-gaussian-bound}. If $0\leq x\leq\nu/2$, then
		$x^2/j_{\nu,r}^2\leq1/4$. Thus, by
		$\log(1-u)=-u+O(u^2)$, the Weierstrass product, and the two Rayleigh
		identities,
		\begin{align*}
			\log\mathcal J_d(x)
			&=-x^2\sum_{r=1}^{\infty}\frac1{j_{\nu,r}^2}
			+O\left(x^4\sum_{r=1}^{\infty}\frac1{j_{\nu,r}^4}\right)\\
			&=-\frac{x^2}{4(\nu+1)}+O\left(\frac{x^4}{\nu^3}\right)
			=-\frac{x^2}{2d+4}+O\left(\frac{x^4}{d^3}\right),
		\end{align*}
		which proves \eqref{eq:bessel-local-expansion}.
		
		For the two real-axis estimates, recall first that for $\nu\geq0$ the
		standard inequality $|J_\nu(x)|\leq1$ can be found in
		\cite[p.~49]{Watson1944}. When $x\geq2\nu$, use the Schl\"afli
		representation \cite[pp.~19--21]{Watson1944}:
		\[
		J_\nu(x)
		=\frac1\pi\int_0^\pi\cos(x\sin t-\nu t)\,\mathrm{d}t
		-\frac{\sin(\pi\nu)}{\pi}
		\int_0^\infty e^{-x\sinh t-\nu t}\,\mathrm{d}t.
		\]
		Set $\phi(t)=x\sin t-\nu t$ and divide $[0,\pi]$ into
		\[
		I_1=[0,\pi/6],\qquad
		I_2=[\pi/6,5\pi/6],\qquad
		I_3=[5\pi/6,\pi].
		\]
		On $I_1$,
		\[
		\phi'(t)\geq\frac{\sqrt3-1}{2}x,
		\]
		while on $I_3$ one has $|\phi'(t)|\geq\sqrt3x/2$. Since $\phi'(t)$ is
		monotone, the first-derivative form of the van der Corput lemma gives
		\[
		\left|\int_{I_1}e^{i\phi(t)}\,\mathrm{d}t\right|
		+\left|\int_{I_3}e^{i\phi(t)}\,\mathrm{d}t\right|=O(x^{-1}).
		\]
		For $t\in I_2$, we have
		$|\phi''(t)|=x\sin t\geq x/2$, so the second-derivative form gives
		\[
		\left|\int_{I_2}e^{i\phi(t)}\,\mathrm{d}t\right|=O(x^{-1/2}).
		\]
		All constants here are independent of $\nu$ and $x$. The nonoscillatory
		term in the Schl\"afli representation satisfies
		\[
		\int_0^\infty e^{-x\sinh t-\nu t}\,\mathrm{d}t
		\leq\int_0^\infty e^{-xt}\,\mathrm{d}t=x^{-1}.
		\]
		Since $x\geq2\nu=d\geq1$, the $O(x^{-1})$ term is absorbed by
		$O(x^{-1/2})$. Hence $|J_\nu(x)|\leq Cx^{-1/2}$ holds uniformly for
		$x\geq2\nu$.
		Substituting these two real-axis estimates into
		\eqref{eq:normalized-bessel} proves the last two assertions of the lemma.
	\end{proof}
	
	\subsection{Gaussian approximation}
	
	In this subsection, $\ell_d,q_d,A_d$ have the meanings given above, and
	$R=R_d$.
	
	\begin{lemma}[High-frequency tail]\label{lem:common-tail}
		Suppose that
		\[
		\ell_d\longrightarrow\infty,
		\qquad
		\ell_d=o(\sqrt d).
		\]
		Let
		\[
		K_d=\frac{d^2}{16\pi^2R^2}\sim\frac{d}{4\ell_d}.
		\]
		Then
		\[
		\mathcal J_d(2\pi R)=q_d(1+o(1))
		\]
		and
		\[
		\sum_{|\bbf m|^2>K_d}
		|\mathcal J_d(2\pi R|\bbf m|)|^2=o(dq_d^2).
		\]
		If, in addition, $e^{-A_d}=o(\log d)$, then for every fixed $H>0$,
		\[
		e^{He^{-A_d}}
		\sum_{|\bbf m|^2>K_d}
		|\mathcal J_d(2\pi R|\bbf m|)|^2\longrightarrow0.
		\]
	\end{lemma}
	
	\begin{proof}
		Since $(2\pi R)^2=(d+2)\ell_d$ and $\ell_d=o(\sqrt d)$, we have
		$2\pi R\leq d/4$ for all sufficiently large $d$. By
		\eqref{eq:bessel-local-expansion},
		\begin{align*}
			\mathcal J_d(2\pi R)
			&=q_d\exp\left(O\left(\frac{(2\pi R)^4}{d^3}\right)\right)\\
			&=q_d\exp\left(O\left(\frac{\ell_d^2}{d}\right)\right)
			=q_d(1+o(1)).
		\end{align*}
		
		We now treat the tail. Write
		$N_d(u)=\#\{\bbf m\in\mathbb Z^d:|\bbf m|^2\leq u\}$. For
		$1\leq u\leq2d$, put $n=\lceil u\rceil$. Since integer coordinates satisfy
		$\sum_j|m_j|\leq\sum_jm_j^2\leq n$, the possible absolute-value vectors are
		counted by $\binom{d+n}{n}$, and there are at most $2^n$ sign choices. Thus
		\[
		N_d(u)
		\leq2^{\lceil u\rceil}
		\binom{d+\lceil u\rceil}{\lceil u\rceil}
		\leq
		\left(\frac{2\mathrm e(d+\lceil u\rceil)}{\lceil u\rceil}
		\right)^{\lceil u\rceil}.
		\]
		The unit cubes centered at these lattice points are disjoint and are contained
		in the ball of radius $\sqrt u+\sqrt d/2$. Comparing volumes also gives
		\[
		N_d(u)
		\leq\frac{\pi^{d/2}(\sqrt u+\sqrt d/2)^d}
		{\Gamma(1+\frac d2)}
		\leq\left(C\left(1+\sqrt{\frac ud}\right)\right)^d.
		\]
		
		By Stirling's formula and the real-axis estimates in
		Lemma~\ref{lem:bessel}, for $k\geq1$,
		\[
		|\mathcal J_d(2\pi R\sqrt{k})|^2
		\leq Cd\left(\frac{d}{2\pi\mathrm eR\sqrt{k}}\right)^d.
		\]
		If $2\pi R\sqrt{k}\geq d$, then also
		\[
		|\mathcal J_d(2\pi R\sqrt{k})|^2
		\leq Cd\left(\frac{d}{2\pi\mathrm eR\sqrt{k}}\right)^d
		(2\pi R\sqrt{k})^{-1}.
		\]
		
		The assumptions imply
		\[
		K_d\longrightarrow\infty,
		\qquad K_d<\frac d{\sqrt{\ell_d}}<d
		\]
		for all sufficiently large $d$. Cover $(K_d,d]$ by dyadic intervals
		$(s,2s]$, starting with $s=K_d$ and truncating the last interval if necessary.
		If
		$K_d\leq s\leq d/\sqrt{\ell_d}$, the combinatorial bound above gives,
		uniformly in this range,
		\[
		\log N_d(2s)
		\ll s\log\left(2+\frac ds\right)
		\ll \frac{d\log\ell_d}{\sqrt{\ell_d}}=o(d).
		\]
		Moreover,
		$d/(2\pi\mathrm eR\sqrt s)\leq2/\mathrm e$. Hence
		\begin{align*}
			\sum_{s<|\bbf m|^2\leq2s}
			|\mathcal J_d(2\pi R|\bbf m|)|^2
			&\leq N_d(2s)Cd
			\left(\frac{d}{2\pi\mathrm eR\sqrt s}\right)^d\\
			&\leq \exp(o(d))Cd\left(\frac2{\mathrm e}\right)^d\\
			&=\exp\left(-\left(\log\frac{\mathrm e}{2}+o(1)\right)d\right).
		\end{align*}
		If $d/\sqrt{\ell_d}<s\leq d$, the geometric bound gives
		$\log N_d(2s)=O(d)$, while
		$d/(2\pi\mathrm eR\sqrt s)\leq C\ell_d^{-1/4}$. Therefore
		\begin{align*}
			\sum_{s<|\bbf m|^2\leq2s}
			|\mathcal J_d(2\pi R|\bbf m|)|^2
			&\leq \exp(O(d))Cd\left(C\ell_d^{-1/4}\right)^d\\
			&=\exp\left(-\frac14d\log\ell_d+O(d)\right)\\
			&\leq\exp\left(-\frac18d\log\ell_d\right)
		\end{align*}
		for all sufficiently large $d$. Finally, cover $(d,\infty)$ by intervals
		$(s,2s]$ with $s=2^jd$ and $j\geq0$. On these intervals
		$2\pi R\sqrt{k}\geq2\pi R\sqrt d\geq d$ for all sufficiently large $d$,
		so the strengthened real-axis estimate applies. Moreover,
		then $N_d(2s)\leq(C2^{j/2})^d$ and
		\begin{align*}
			\sum_{s<|\bbf m|^2\leq2s}
			|\mathcal J_d(2\pi R|\bbf m|)|^2
			&\leq (C2^{j/2})^dCd
			\left(\frac{C}{2^{j/2}\sqrt{\ell_d}}\right)^d
			(2\pi R\sqrt d)^{-1}2^{-j/2}\\
			&\leq Cd\left(\frac{C}{\sqrt{\ell_d}}\right)^d
			(2\pi R\sqrt d)^{-1}2^{-j/2}.
		\end{align*}
		
		There are $O(\log\ell_d)$ intervals of the first two types in total, while
		summing the third type over $j$ introduces only an absolute constant. Let
		$c_0=\log(\mathrm e/2)>0$. Since
		$A_d=\ell_d-\log d=o(\sqrt d)$ and $dq_d^2=e^{-A_d}$, the total
		contribution of the first type satisfies
		\[
		\frac{O(\log\ell_d)e^{-(c_0+o(1))d}}{dq_d^2}
		=\exp\left(-(c_0+o(1))d+A_d+O(\log\log\ell_d)\right)
		\longrightarrow0.
		\]
		The logarithm of the corresponding ratio for the second type is at most
		\[
		-\frac18d\log\ell_d+A_d+O(\log\log\ell_d),
		\]
		and for the third type it is at most
		\[
		d\log\left(\frac{C}{\sqrt{\ell_d}}\right)+A_d+O(\log d).
		\]
		Both tend to $-\infty$, proving the tail estimate. More precisely, after
		summing the dyadic intervals, the first range is bounded by
		$\exp(-(c_0+o(1))d)$, the second by
		$\exp(-\frac18d\log\ell_d+O(\log\log\ell_d))$, and the third by
		\[
		Cd\left(\frac{C}{\sqrt{\ell_d}}\right)^d
		(2\pi R\sqrt d)^{-1}\sum_{j\geq0}2^{-j/2}.
		\]
		Since $\ell_d\to\infty$, the full tail is at most
		$\exp(-cd+o(d))$ for some fixed $c>0$. Finally, if
		$e^{-A_d}=o(\log d)$, then for fixed $H>0$,
		\[
		-cd+o(d)+He^{-A_d}=-cd+o(d),
		\]
		which proves the exponentially weighted assertion.
	\end{proof}
	
	\begin{proposition}[$\theta$-series limit theorem]\label{prop:theta-approx}
		Suppose that $0<q_d<1$, $q_d\to0$, and $dq_d^2=o(\sqrt d)$. For
		$\bbf x\in[0,1)^d$, define
		\[
		T_d(\bbf x)=\sum_{\bbf m\in\mathbb Z^d}
		e(-\bbf x\cdot\bbf m)q_d^{|\bbf m|^2}.
		\]
		Then, for any real numbers $a<b$,
		\[
		\lim_{d\to\infty}\operatorname{meas}\left\{\bbf x\in[0,1)^d:
		a<\frac{\log T_d(\bbf x)+dq_d^2}{\sqrt{2dq_d^2}}\leq b\right\}
		=\frac{1}{\sqrt{2\pi}}\int_a^b e^{-u^2/2}\,\mathrm{d}u.
		\]
		If $dq_d^2\to0$, then
		\[
		\lim_{d\to\infty}\operatorname{meas}\left\{\bbf x\in[0,1)^d:
		a<\frac{T_d(\bbf x)-1}{\sqrt{2dq_d^2}}\leq b\right\}
		=\frac{1}{\sqrt{2\pi}}\int_a^b e^{-u^2/2}\,\mathrm{d}u,
		\]
		and
		\[
		\int_{[0,1)^d}|T_d(\bbf x)-1|^2\,\mathrm{d}\bbf x\sim2dq_d^2.
		\]
	\end{proposition}
	
	\begin{proof}
		Write
		$\bbf x=(x_1,\ldots,x_d)$. Absolute convergence gives
		\[
		T_d(\bbf x)=\prod_{j=1}^dg_d(x_j),
		\qquad
		g_d(x)=\sum_{n\in\mathbb Z}q_d^{n^2}e(-nx).
		\]
		Since $q_d\to0$, uniformly in $x$,
		\[
		g_d(x)=1+2q_d\cos(2\pi x)+O(q_d^4).
		\]
		In particular,
		\[
		g_d(x)\geq1-2q_d-Cq_d^4>0
		\]
		for all sufficiently large $d$, uniformly in $x$. From
		$\log(1+u)=u-u^2/2+O(u^3)$,
		\[
		\log g_d(x)=2q_d\cos(2\pi x)
		-2q_d^2\cos^2(2\pi x)+O(q_d^3).
		\]
		Thus
		\[
		\int_0^1\log g_d(x)\,\mathrm{d}x=-q_d^2+O(q_d^3),
		\]
		while the variance of $\log g_d$ is $2q_d^2+O(q_d^3)$ and its third
		centered absolute moment is $O(q_d^3)$. In particular, this variance is at
		least $q_d^2$ for all sufficiently large $d$. Under Lebesgue measure on
		$[0,1)^d$, the coordinate functions are independent and uniformly distributed
		on $[0,1)$. Hence $\log g_d(x_1),\ldots,\log g_d(x_d)$ are independent and
		identically distributed. The corresponding Lyapunov ratio is
		\[
		\frac{dO(q_d^3)}{\bigl(d(2q_d^2+O(q_d^3))\bigr)^{3/2}}
		=O(d^{-1/2}),
		\]
		so the Lyapunov central limit theorem for triangular arrays applies. The
		variance used in the proposition differs from the exact variance by a
		relative factor $1+O(q_d)=1+o(1)$. The replacement of the exact mean by
		$-dq_d^2$ is also negligible on the stated scale, since
		\[
		\frac{d\left|\int_0^1\log g_d(x)\,\mathrm{d}x+q_d^2\right|}
		{\sqrt{2dq_d^2}}
		=O\left(\frac{dq_d^3}{\sqrt{dq_d^2}}\right)
		=O\left(\frac{dq_d^2}{\sqrt d}\right)
		\longrightarrow0,
		\]
		where the limit follows from $dq_d^2=o(\sqrt d)$. This proves the
		logarithmic limit in the proposition.
		
		If $dq_d^2\to0$, the preceding logarithmic limit gives
		\[
		\frac{\log T_d+dq_d^2}{\sqrt{2dq_d^2}}
		\mathrel{\Rightarrow}N(0,1).
		\]
		Since $dq_d^2=o(\sqrt{dq_d^2})$, it follows that
		\[
		\frac{\log T_d}{\sqrt{dq_d^2}}
		\quad\text{forms a tight family, and}\quad
		\log T_d\longrightarrow0
		\quad\text{in measure}.
		\]
		Using $|e^u-1-u|\leq |u|^2e^{|u|}$, we obtain
		\[
		\frac{|T_d-1-\log T_d|}{\sqrt{dq_d^2}}
		\leq
		\frac{|\log T_d|}{\sqrt{dq_d^2}}
		\,|\log T_d|e^{|\log T_d|}
		\longrightarrow0\quad\text{in measure}.
		\]
		Since also
		$dq_d^2/\sqrt{dq_d^2}\to0$, the asserted
		distributional limit for $T_d-1$ follows. Orthogonality of the exponential
		system further gives
		\[
		\int_{[0,1)^d}|T_d(\bbf x)-1|^2\,\mathrm{d}\bbf x
		=\left(\sum_{n\in\mathbb Z}q_d^{2n^2}\right)^d-1
		=\left(1+2q_d^2+O(q_d^8)\right)^d-1.
		\]
		Writing $s_d=dq_d^2\to0$, we have
		\[
		d\log\left(1+2q_d^2+O(q_d^8)\right)
		=2s_d+O(dq_d^4+dq_d^8)=2s_d+o(s_d).
		\]
		Therefore the preceding second moment is asymptotic to
		$2s_d=2dq_d^2$.
	\end{proof}
	
	\begin{proposition}[Fourier coefficient comparison]\label{prop:fourier-comparison}
		Under the assumptions of Lemma~\ref{lem:common-tail}, suppose in addition
		that $dq_d^2=o(\log d)$. If
		$\liminf_{d\to\infty}dq_d^2>0$, then for every fixed
		$\varepsilon>0$,
		\[
		e^{2(1+\varepsilon)dq_d^2}
		\sum_{\bbf m\in\mathbb Z^d}
		\left|\mathcal J_d(2\pi R|\bbf m|)
		-q_d^{|\bbf m|^2}\right|^2\longrightarrow0.
		\]
		If $dq_d^2\to0$, then the same sum without the exponential factor is
		$o(dq_d^2)$.
	\end{proposition}
	
	\begin{proof}
		Let
		$r_d(k)=\#\{\bbf m\in\mathbb Z^d:|\bbf m|^2=k\}$. When
		$dq_d^2\to0$, the error contributed by the frequencies with
		$|\bbf m|^2=1$ is
		\[
		2d\left|\mathcal J_d(2\pi R)-q_d\right|^2=o(dq_d^2),
		\]
		while
		\[
		\sum_{k\geq2}r_d(k)q_d^{2k}
		=\left(1+2q_d^2+O(q_d^8)\right)^d-1-2dq_d^2
		=O(dq_d^8+d^2q_d^4)=O(d^2q_d^4)=o(dq_d^2).
		\]
		For $2\leq k\leq K_d$, \eqref{eq:bessel-gaussian-bound} gives
		$|\mathcal J_d(2\pi R\sqrt{k})|\leq q_d^k$, and hence
		\[
		\sum_{2\leq|\bbf m|^2\leq K_d}
		\left|\mathcal J_d(2\pi R|\bbf m|)
		-q_d^{|\bbf m|^2}\right|^2
		\leq4\sum_{k\geq2}r_d(k)q_d^{2k}=o(dq_d^2).
		\]
		For $k>K_d$, the inequality $|a-b|^2\leq2|a|^2+2|b|^2$ gives
		\begin{align*}
			\sum_{|\bbf m|^2>K_d}
			\left|\mathcal J_d(2\pi R|\bbf m|)
			-q_d^{|\bbf m|^2}\right|^2
			&\leq2\sum_{|\bbf m|^2>K_d}
			|\mathcal J_d(2\pi R|\bbf m|)|^2
			+2\sum_{k>K_d}r_d(k)q_d^{2k}\\
			&=o(dq_d^2),
		\end{align*}
		by Lemma~\ref{lem:common-tail}; the Gaussian term is bounded by the total
		$k\geq2$ contribution above (and $K_d\to\infty$).
		Together with the frequencies $|\bbf m|^2=1$, this proves the estimate
		when $dq_d^2\to0$.
		
		Now suppose that $\liminf_{d\to\infty}dq_d^2>0$, and fix
		$\varepsilon>0$. Set
		\[
		B_d=\left(\frac{\sqrt d}{dq_d^2\ell_d}\right)^{1/2},
		\qquad
		M_d=\left\lceil B_d\,dq_d^2\right\rceil.
		\]
		Since $dq_d^2=e^{-A_d}=o(\log d)$ and is bounded below in the present
		case, $|A_d|=O(\log\log d)$ and $\ell_d\sim\log d$. In addition,
		\[
		B_d^2=\frac{\sqrt d}{dq_d^2\ell_d}\longrightarrow\infty,
		\]
		because $dq_d^2\ell_d=o((\log d)^2)=o(\sqrt d)$.
		Furthermore,
		\[
		\frac{M_d^2\ell_d^2}{d}
		\ll\frac{dq_d^2\ell_d}{\sqrt d}+\frac{\ell_d^2}{d}
		\longrightarrow0,
		\]
		and
		\[
		\frac{M_d}{K_d}
		\ll\frac{\sqrt{dq_d^2\ell_d}}{d^{3/4}}+\frac{\ell_d}{d}
		\longrightarrow0.
		\]
		Thus $M_d<K_d/4$ for all sufficiently large $d$. Consequently, for
		$k\leq M_d$,
		\[
		\left|\mathcal J_d(2\pi R\sqrt{k})-q_d^k\right|
		\ll q_d^k\frac{k^2\ell_d^2}{d}.
		\]
		Since $q_d^2=d^{-1}e^{-A_d}=o(1)$,
		\[
		\left(\sum_{n\in\mathbb Z}q_d^{2n^2}e^{n^2}\right)^d
		=\exp\left(2\mathrm e\,dq_d^2+O(dq_d^4+dq_d^8)\right)
		=\exp\left(2\mathrm e\,dq_d^2+o(1)\right).
		\]
		We therefore have the Gaussian tail estimate
		\[
		\sum_{k>M_d}r_d(k)q_d^{2k}
		\leq e^{-M_d}
		\left(\sum_{n\in\mathbb Z}q_d^{2n^2}e^{n^2}\right)^d
		\leq\exp\left(-(B_d-2\mathrm e)dq_d^2+o(1)\right).
		\]
		For the low-frequency range $k\leq M_d$, using $k^4\ll e^k$ gives
		\begin{align*}
			&e^{2(1+\varepsilon)dq_d^2}
			\sum_{|\bbf m|^2\leq M_d}
			\left|\mathcal J_d(2\pi R|\bbf m|)
			-q_d^{|\bbf m|^2}\right|^2\\
			&\quad\leq
			\frac{C\ell_d^4}{d^2}
			\exp\left((2\mathrm e+2+2\varepsilon)dq_d^2+o(1)\right).
		\end{align*}
		The logarithm of the right-hand side is
		\[
		-2\log d+O(\log\log d)
		+(2\mathrm e+2+2\varepsilon)dq_d^2+o(1)
		=-2\log d+o(\log d),
		\]
		so the low-frequency contribution tends to zero.
		
		For $M_d<k\leq K_d$,
		$|\mathcal J_d(2\pi R\sqrt{k})|\leq q_d^k$. Hence the
		intermediate-frequency contribution satisfies
		\begin{align*}
			&e^{2(1+\varepsilon)dq_d^2}
			\sum_{M_d<|\bbf m|^2\leq K_d}
			\left|\mathcal J_d(2\pi R|\bbf m|)
			-q_d^{|\bbf m|^2}\right|^2\\
			&\quad\leq4e^{2(1+\varepsilon)dq_d^2}
			\sum_{k>M_d}r_d(k)q_d^{2k}\\
			&\quad\leq4\exp\left(
			-\bigl(B_d-2\mathrm e-2-2\varepsilon\bigr)dq_d^2+o(1)\right)
			\longrightarrow0,
		\end{align*}
		because $B_d\to\infty$ and $\liminf dq_d^2>0$.
		
		For the high-frequency range $k>K_d$, the inequality
		$|a-b|^2\leq2|a|^2+2|b|^2$ gives
		\begin{align*}
			&e^{2(1+\varepsilon)dq_d^2}
			\sum_{|\bbf m|^2>K_d}
			\left|\mathcal J_d(2\pi R|\bbf m|)
			-q_d^{|\bbf m|^2}\right|^2\\
			&\quad\leq2e^{2(1+\varepsilon)dq_d^2}
			\sum_{|\bbf m|^2>K_d}
			|\mathcal J_d(2\pi R|\bbf m|)|^2
			+2e^{2(1+\varepsilon)dq_d^2}
			\sum_{k>K_d}r_d(k)q_d^{2k}.
		\end{align*}
		The first term on the right tends to zero by the strengthened conclusion
		of Lemma~\ref{lem:common-tail} with $H=2(1+\varepsilon)$, while the second is
		at most
		\[
		2\exp\left(
		-\bigl(B_d-2\mathrm e-2-2\varepsilon\bigr)dq_d^2+o(1)\right)
		\longrightarrow0.
		\]
		Thus the high-frequency contribution also tends to zero.
		
		Adding the low-, intermediate-, and high-frequency estimates yields
		\[
		e^{2(1+\varepsilon)dq_d^2}
		\sum_{\bbf m\in\mathbb Z^d}
		\left|\mathcal J_d(2\pi R|\bbf m|)
		-q_d^{|\bbf m|^2}\right|^2
		\longrightarrow0.
		\]
		This proves the required Fourier coefficient estimate.
	\end{proof}
	
	\subsection{Proofs of the two limit theorems}
	
	\begin{proof}[Proof of Theorem~\ref{thm:logarithmic}]
		By assumption, $e^{-A_d}$ has a positive lower bound and
		$|A_d|=O(\log\log d)$. Hence $\ell_d\sim\log d$. Moreover,
		\[
		q_d^2=\frac{e^{-A_d}}d\longrightarrow0,
		\qquad dq_d^2=e^{-A_d}=o(\log d)=o(\sqrt d).
		\]
		Thus
		Lemma~\ref{lem:common-tail} and
		Propositions~\ref{prop:theta-approx} and~\ref{prop:fourier-comparison} apply.
		In particular, $T_d$ satisfies
		\[
		\lim_{d\to\infty}\operatorname{meas}\left\{\bbf x\in[0,1)^d:
		a<\frac{\log T_d(\bbf x)+e^{-A_d}}{\sqrt{2e^{-A_d}}}\leq b
		\right\}
		=\frac{1}{\sqrt{2\pi}}\int_a^b e^{-u^2/2}\,\mathrm{d}u.
		\]
		
		By Parseval's identity and Proposition~\ref{prop:fourier-comparison}, for
		every fixed $\varepsilon>0$,
		\[
		e^{2(1+\varepsilon)e^{-A_d}}
		\int_{[0,1)^d}\left|
		\frac{f_d(\bbf x)}{\mu_d}-T_d(\bbf x)
		\right|^2\,\mathrm{d}\bbf x\longrightarrow0.
		\]
		Fix $M>0$ and $\eta>0$. On the set
		\[
		\left|\frac{\log T_d(\bbf x)+e^{-A_d}}
		{\sqrt{2e^{-A_d}}}\right|\leq M,
		\]
		we have
		\[
		T_d(\bbf x)\geq
		\exp\left(-e^{-A_d}-M\sqrt{2e^{-A_d}}\right).
		\]
		Chebyshev's inequality therefore gives
		\[
		\operatorname{meas}\left\{\bbf x:\left|
		\frac{f_d(\bbf x)}{\mu_dT_d(\bbf x)}-1\right|>\eta,
		\left|\frac{\log T_d(\bbf x)+e^{-A_d}}
		{\sqrt{2e^{-A_d}}}\right|\leq M\right\}
		\leq\frac{e^{2e^{-A_d}+2M\sqrt{2e^{-A_d}}}}{\eta^2}
		\int_{[0,1)^d}\left|
		\frac{f_d(\bbf x)}{\mu_d}-T_d(\bbf x)\right|^2\,\mathrm{d}\bbf x.
		\]
		For fixed $M$, the right-hand side tends to zero. Indeed, it equals
		\[
		\eta^{-2}e^{-2\varepsilon e^{-A_d}+2M\sqrt{2e^{-A_d}}}
		\left[e^{2(1+\varepsilon)e^{-A_d}}
		\left\|\frac{f_d}{\mu_d}-T_d\right\|_2^2\right],
		\]
		where the exponential prefactor is uniformly bounded as a function of
		$e^{-A_d}>0$, and the bracketed factor tends to zero. The limit theorem for
		$T_d$ also shows that the
		standardized logarithms form a tight family. Moreover,
		\[
		\operatorname{meas}\left\{\left|\frac{f_d}{\mu_dT_d}-1\right|>\eta\right\}
		\leq\operatorname{meas}\left\{
		\left|\frac{\log T_d+e^{-A_d}}{\sqrt{2e^{-A_d}}}\right|>M\right\}
		+\operatorname{meas}\left\{
		\left|\frac{f_d}{\mu_dT_d}-1\right|>\eta,
		\left|\frac{\log T_d+e^{-A_d}}{\sqrt{2e^{-A_d}}}\right|\leq M
		\right\}.
		\]
		Letting first $d\to\infty$ and then $M\to\infty$ yields
		\[
		\frac{f_d}{\mu_dT_d}\longrightarrow1\quad\text{in measure}.
		\]
		Since
		\[
		R_d^2=\frac{(d+2)\ell_d}{4\pi^2}>\frac d4
		\]
		for all sufficiently large $d$, every point of $[0,1)^d$ lies within
		distance $\sqrt d/2<R_d$ of a point of $\mathbb Z^d$. Thus
		$f_d(\bbf x)>0$ everywhere for large $d$. Because $T_d>0$, the continuous
		mapping theorem gives
		$\log(f_d/(\mu_dT_d))\to0$ in measure. Finally,
		$\liminf e^{-A_d}>0$, so
		\[
		\frac{\log(f_d/(\mu_dT_d))}{\sqrt{2e^{-A_d}}}
		\longrightarrow0\quad\text{in measure}.
		\]
		Finally,
		\[
		\log\left(\frac{f_d}{\mu_d}\right)
		=\log T_d+\log\left(\frac{f_d}{\mu_dT_d}\right).
		\]
		Combining this identity with the limit for $T_d$ and applying Slutsky's theorem
		proves \eqref{eq:logarithmic-clt}.
	\end{proof}
	
	\begin{proof}[Proof of the lognormal corollary]
		If $A_d\to\tau$, Theorem~\ref{thm:logarithmic} gives
		\[
		\log(f_d/\mu_d)\mathrel{\Rightarrow}
		N(-e^{-\tau},2e^{-\tau}).
		\]
		The continuous mapping theorem applied to the exponential function gives the
		stated lognormal limit and its density formula.
	\end{proof}
	
	\begin{proof}[Proof of Theorem~\ref{thm:concentration}]
		By definition,
		\[
		\ell_d=\log d+A_d,
		\qquad
		q_d^2=\frac{e^{-A_d}}{d}.
		\]
		The conditions in \eqref{eq:concentration-assumptions} imply that
		$\ell_d\to\infty$, $\ell_d=o(\sqrt d)$, $q_d\to0$, and
		$e^{-A_d}=dq_d^2\to0$. Hence the hypotheses of
		Lemma~\ref{lem:common-tail} and
		Propositions~\ref{prop:theta-approx} and~\ref{prop:fourier-comparison} all hold.
		Retain the notation $T_d$ from
		Proposition~\ref{prop:theta-approx}. By
		\eqref{eq:fourier-expansion}, orthogonality of the exponential system, and
		Proposition~\ref{prop:fourier-comparison},
		\[
		\int_{[0,1)^d}\left|
		\frac{f_d(\bbf x)}{\mu_d}-T_d(\bbf x)
		\right|^2\,\mathrm{d}\bbf x=o(e^{-A_d}).
		\]
		Thus, for every $\eta>0$,
		\[
		\operatorname{meas}\left\{\bbf x\in[0,1)^d:\left|
		\frac{f_d(\bbf x)/\mu_d-T_d(\bbf x)}{\sqrt{2e^{-A_d}}}
		\right|>\eta\right\}
		\leq\frac{1}{2\eta^2e^{-A_d}}
		\int_{[0,1)^d}\left|
		\frac{f_d(\bbf x)}{\mu_d}-T_d(\bbf x)
		\right|^2\,\mathrm{d}\bbf x
		\longrightarrow0.
		\]
		Combining this with the distributional limit in
		Proposition~\ref{prop:theta-approx} proves
		\eqref{eq:concentration-clt}. The standardized variables are tight, and
		$\sqrt{2e^{-A_d}}\to0$; hence $f_d/\mu_d\to1$ in measure.
		
		For the variance asymptotic, set
		\[
		U_d(\bbf x)=\frac{f_d(\bbf x)}{\mu_d}-1,
		\qquad
		V_d(\bbf x)=T_d(\bbf x)-1.
		\]
		The $L^2$ estimate above gives
		\[
		\|U_d-V_d\|_2=o(e^{-A_d/2}).
		\]
		Proposition~\ref{prop:theta-approx} also gives
		$\|V_d\|_2^2\sim2e^{-A_d}$, and hence
		$\|U_d\|_2=O(e^{-A_d/2})$. Therefore
		\[
		\left|\|U_d\|_2^2-\|V_d\|_2^2\right|
		\leq
		\|U_d-V_d\|_2\bigl(\|U_d\|_2+\|V_d\|_2\bigr)
		=o(e^{-A_d}).
		\]
		Finally, $\int_{[0,1)^d}U_d\,\mathrm{d}\bbf x=0$, so
		$\|U_d\|_2^2=\operatorname{Var}(f_d/\mu_d)$. Hence
		\[
		\int_{[0,1)^d}\left|\frac{f_d(\bbf x)}{\mu_d}-1\right|^2
		\,\mathrm{d}\bbf x
		\sim2e^{-A_d},
		\]
		which is \eqref{eq:concentration-var}.
	\end{proof}
	
	\section*{Declaration on the Use of Generative AI}
	
	During the preparation of this manuscript, the author used ChatGPT (OpenAI)
	to translate an initial Chinese-language draft into English and to assist with
	language editing and stylistic refinement. All AI-assisted text was subsequently
	reviewed and revised by the author, who takes full responsibility for the
	accuracy, originality, and integrity of the manuscript.

\end{document}